\documentclass{amsart}
\usepackage[latin1]{inputenc}
\usepackage[T1]{fontenc}
\usepackage{mathrsfs,amssymb,amsthm,amsmath,verbatim}

\title{Indestructibility of compact spaces}
\author[R. R. Dias]{Rodrigo R. Dias$^1$}
\thanks{$^1$ Supported by Capes (BEX-1088-11-4)
 and CNPq (143407/2010-4)}
\address{Instituto de Matem\'atica e Estat\'istica,
Universidade de S\~ao Paulo, Caixa Postal 66281,
S\~ao Paulo, SP, 05314-970, Brazil}
\address{Capes Foundation, Ministry of Education of Brazil,
Caixa Postal 250, Bras\'ilia, DF, 70040-020, Brazil}
\email{roque@ime.usp.br}
\author[F. D. Tall]{Franklin D. Tall$^2$}
\thanks{$^2$ Supported by NSERC Grant A-7354 and FAPESP
  (2011/18046-1)}
\address{Department of Mathematics,
University of Toronto,
Toronto, ON, M5S 2E4, Canada}
\email{f.tall@utoronto.ca}



\begin{document}

\newtheorem{defin}{Definition}[section]

\newtheorem{prop}[defin]{Proposition}

\newtheorem{prob}[defin]{Problem}

\newtheorem{lemma}[defin]{Lemma}

\newtheorem{corol}[defin]{Corollary}

\newtheorem{example}[defin]{Example}

\newtheorem{thm}[defin]{Theorem}

\newtheorem{rem}[defin]{Remark}

\begin{abstract}

In this article we investigate which compact spaces remain compact
under countably closed forcing.
We prove that, assuming the Continuum Hypothesis, the natural
generalizations to $\omega_1$-sequences of the selection principle and
topological game versions of the Rothberger property are not
equi\-va\-lent, even for compact spaces. We also show that Tall and
Usuba's ``$\aleph_1$-Borel Conjecture'' is equiconsistent with the
existence of an inaccessible cardinal.

\end{abstract}

\maketitle

\section{Introduction}

The question of whether the Lindel\"of property of topological spaces
is preserved by countably closed forcing is of interest in conjunction
with A. V. Arhangel'ski\u\i's classic problem of whether Lindel\"of
$T_2$ spaces with points $G_\delta$ have cardinality not exceeding the
continuum \cite{arh}. For a survey of this problem, see \cite{tall95}.

\begin{defin}[Tall \cite{tall95}]
\label{indestructible}

A Lindel\"of space is \emph{indestructible}
(or \emph{indestructibly Lindel\"of})
if the topology it generates in any countably closed forcing extension
is Lindel\"of.

\end{defin}

\begin{thm}[Tall \cite{tall95}]
\label{supercompact}

L\'evy-collapse a supercompact cardinal to $\omega_2$ with coun\-table
conditions. Then every indestructible Lindel\"of space with
points $G_\delta$ has cardinality $\le\aleph_1$.

\end{thm}

This was later improved by M. Scheepers \cite{sch10}, who replaced
``supercompact'' by ``measurable''.
``Points $G_\delta$'' was then improved to ``pseudocharacter
$\le\aleph_1$'' by Tall and T. Usuba in \cite{usuba}.

Scheepers and Tall \cite{schtall}
noticed that indestructibility of a Lindel\"of space is equi\-va\-lent
to player One not having a winning strategy in an $\omega_1$-length
generalization of the \emph{Rothberger game} introduced by F. Galvin
in \cite{galvin}.

\begin{defin}
\label{game}

Let $X$ be a topological space and $\alpha$ be an ordinal. We denote by
$\mathsf{G}_1^\alpha(\mathcal{O}_X,\mathcal{O}_X)$ the game defined as
follows. In each inning $\xi\in\alpha$, player One
chooses an open cover $\mathcal{U}_\xi$ of $X$, and then player Two
picks $V_\xi\in\mathcal{U}_\xi$. Two wins the play if
$X=\bigcup\{V_\xi:\xi\in\alpha\}$; otherwise, One
is the winner.

\end{defin}

\begin{thm}[Scheepers-Tall \cite{schtall}]
\label{st1}

A Lindel\"of space $X$ is indestructible if and only if One does not
have a winning strategy in the game $\mathsf{G}_1^{\omega_1}(\mathcal
O_X,\mathcal O_X)$.

\end{thm}

\begin{defin}
\label{S1}

Let $X$ be a topological space and $\alpha$ be an
ordinal. $\mathsf{S}_1^\alpha(\mathcal{O}_X,\mathcal{O}_X)$ denotes
the following statement:
``For every sequence $(\mathcal U_\xi)_{\xi\in\alpha}$ of open covers of
$X$, there is a sequence $(U_\xi)_{\xi\in\alpha}$ such that
$U_\xi\in\mathcal U_\xi$ for each $\xi\in\alpha$ and
$X=\bigcup\{U_\xi:\xi\in\alpha\}$''.
\end{defin}

Note that $\mathsf S_1^\omega(\mathcal O_X,\mathcal O_X)$ means that
$X$ is a Rothberger space \cite{roth}.

It is easily seen that the non-existence of a winning strategy for
player One in $\mathsf G_1^\alpha(\mathcal O_X,\mathcal O_X)$ implies
$\mathsf S_1^\alpha(\mathcal O_X,\mathcal O_X)$.
In \cite{pawl}, J. Pawlikowski proved the converse of this
implication for the case $\alpha=\omega$:

\begin{thm}[Pawlikowski \cite{pawl}]
\label{pawl}

A topological space $X$ is Rothberger
if and only if One has
no winning strategy in
$\mathsf{G}_1^\omega(\mathcal{O}_X,\mathcal{O}_X)$.

\end{thm}

We thank Boaz Tsaban for the following observation.
Pawlikowski's result yields the analogous equivalence
for $\omega\le\alpha<\omega_1$:

\begin{corol}
\label{pawl-alpha}

Let $\alpha$ be an infinite countable ordinal. The following are
equivalent for a topological space $X$:

\begin{itemize}

\item[$(a)$]
One does not have a winning strategy in $\mathsf G_1^\alpha(\mathcal
O_X,\mathcal O_X)$;

\item[$(b)$]
$\mathsf S_1^\alpha(\mathcal O_X,\mathcal O_X)$;

\item[$(c)$]
$X$ is Rothberger.

\end{itemize}

\end{corol}

\begin{proof}
We have already pointed out that $(a)\rightarrow(b)$. The equivalence
between $(b)$ and $(c)$ is immediate since
$|\alpha|=|\omega|$. Finally, if every strategy for One
in the game $\mathsf G_1^\omega(\mathcal O_X,\mathcal
O_X)$ can be defeated, then clearly the same holds in the longer
game $\mathsf G_1^\alpha(\mathcal O_X,\mathcal O_X)$; therefore, $(c)$
implies $(a)$ in view of Theorem \ref{pawl}.
\end{proof}

In light of Theorems \ref{st1} and \ref{pawl}, Scheepers and Tall
asked whether the cor\-res\-ponding equivalence would also hold for
the case $\alpha=\omega_1$ \cite{schtall,tallqa}:

\begin{prob}[Scheepers-Tall \cite{schtall}, Tall \cite{tallqa}]
\label{st?}

Is indestructibility of a Lindel\"of space $X$ equivalent to $\mathsf
S_1^{\omega_1}(\mathcal O_X,\mathcal O_X)$?

\end{prob}

We partially answer their question by showing (in Section
\ref{sectionlex}):

\begin{thm}
\label{lex}

There is a compact $T_2$ destructible space that, assuming the
Continuum Hypothesis, satisfies $\mathsf{S}_1^{\omega_1}(\mathcal
O,\mathcal O)$.

\end{thm}

The fact that this example is compact leads us to investigate the
indestructibility of compact spaces. Let us say that a compact space
is \emph{indestructibly compact} if it remains compact
under countably closed forcing.
Notice that:

\begin{lemma}
\label{countcompact}

A compact space is indestructibly Lindel\"of if and only it is
indestructibly compact.

\end{lemma}

\begin{proof}

This follows from the observation that countable compactness is
preserved by countably closed forcing.
\end{proof}

Thus we may refer to \emph{compact indestructible spaces} without
specifying the property that is being preserved.

We shall also prove (in Corollaries \ref{dyadicsize} and
\ref{compactT2}, respectively) that:

\begin{thm}
\label{dyadicintro}

Dyadic spaces of cardinality greater than $\mathfrak{c}$ are
destructible.

\end{thm}

\begin{thm}
\label{noGd->indestr}

Compact $T_2$ spaces in which no point is a $G_\delta$ are
destructible.

\end{thm}

On the other hand, I. Juh\'asz and W. Weiss proved in \cite{scat}
that:

\begin{thm}[Juh\'asz-Weiss \cite{scat}]
\label{jw}

Lindel\"of regular scattered spaces are indestructible.

\end{thm}

In particular,

\begin{example}
\label{1pt}

The one-point compactification of any discrete space is
indestructible.

\end{example}

Here it is worth mentioning that a compact Hausdorff space is
scattered if and only if it is Rothberger (folklore; see e.g.
\cite[Proposition 5.5]{leandro}).

We also have:

\begin{thm}
\label{pointsGd}

The Continuum Hypothesis implies that if a compact $T_2$ space has all
points $G_\delta$, then it is indestructible.

\end{thm}

\begin{proof}

We first recall:

\begin{lemma}[Tall \cite{tall95}]
\label{lealeph1}
Lindel\"of spaces of size $\le\aleph_1$ are indestructible.
\end{lemma}

\begin{proof}
If $X=\{x_\xi:\xi\in\omega_1\}$, then Two gets a winning strategy in
$\mathsf G_1^{\omega_1}(\mathcal O_X,\mathcal O_X)$ by covering the
point $x_\xi$ in the $\xi$-th inning.
Now apply Theorem \ref{st1}.
\end{proof}

Now note that compact Hausdorff spaces with all points $G_\delta$
are first countable and so have cardinality $\le\mathfrak{c}$ by
Arhangel'ski\u\i's Theorem \cite{arh}; thus, if $\mathrm{CH}$ holds,
such spaces are indestructible by Lemma \ref{lealeph1}.
\end{proof}

When some (but not all) points are $G_\delta$'s the situation is less
clear. A new direction for dealing with this problem was provided by
Tall and Usuba \cite{usuba} through a ``one cardinal up'' version
of Borel's Conjecture \cite[p. 123]{borel}.

By results of W. Sierpi\' nski \cite{sierp28} and R. Laver
\cite{laver}, Borel's Conjecture is independent of $\mathrm{ZFC}$. It
is also known (see \cite[Section 5.1]{bonan} and
\cite[Proposition 8]{miller}) to be
equivalent to the statement that a Lindel\"of $T_3$ space is
Rothberger if and only if all of its continuous images in
$[0,1]^\omega$ are countable. In light of this result and Theorems
\ref{st1} and \ref{pawl}, Tall and Usuba formulated the following
statement in analogy to Borel's Conjecture:

\begin{defin}[Tall-Usuba \cite{usuba}]
\label{aleph1BC}

The \emph{$\aleph_1$-Borel Conjecture} is the statement that 
a Lindel\" of $T_3$ space is indestructible if and only if all of its
continuous images in $[0,1]^{\omega_1}$ have cardinality
$\le\aleph_1$.

\end{defin}

Techniques from S. Todor\v cevi\'c
\cite{stevoh} concerning trees and their associated lines enable us to
answer questions of \cite{usuba}; in particular, we show (in Corollary
\ref{equic.BC}):

\begin{thm}
\label{equicBC}

The $\aleph_1$-Borel Conjecture and the existence of an
inaccessible cardinal are equiconsistent.

\end{thm}

\section{Preliminaries}

Our notation and terminology are standard; that said, in this section
we include some definitions and summarize some known results for the
convenience of the reader. For concepts not defined in this section,
the reader is referred to \cite{eng}, \cite{kanamori} and
\cite{kunen}.

A forcing notion $\langle\mathbb P,\le\rangle$ is \emph{countably
closed} if, whenever $\alpha$ is a countable ordinal and
$(p_\xi)_{\xi\in\alpha}$ is a decreasing sequence in $\mathbb P$,
there is $q\in\mathbb P$ such that $q\le p_\xi$ for all
$\xi\in\alpha$. In our forcing arguments, $\mathbf M$ will always
denote the ground model. If $\langle X,\mathcal{T}\rangle$ is a
topological space in $\mathbf M$ and $\mathbb{P}$ is a
forcing notion, in a generic extension $\mathbf M[G]$ by $\mathbb P$
the set $\mathcal{T}$ might fail to be a topology, so we always consider the
corresponding space $\langle X,\widetilde{\mathcal{T}}\rangle$, where
$\widetilde{\mathcal{T}}$ is the topology on $X$ (in the extension
$\mathbf M[G]$) that has $\mathcal{T}$ as an open base.

If $\alpha$ is an ordinal, we write $\lim(\alpha)=\{\xi\in\alpha:\xi$
is a limit ordinal$\}$.
We denote by $\mathfrak{c}$ the cardinality of the
continuum, i.e. $2^{\aleph_0}$; the Continuum Hypothesis
($\mathrm{CH}$) is the assertion
$\mathfrak{c}=\aleph_1$.

A \emph{Dedekind cut} in a linear order $\langle X,<\rangle$ is a set
$A\subseteq X$ such that $y\in A$ whenever $y<x$ for some $x\in
A$. The linear order $\langle X,<\rangle$ is \emph{complete} if every
Dedekind cut has a least upper bound.

For a set $X$ and a cardinal $\kappa$, we write
$[X]^\kappa=\{Y\subseteq X:|Y|=\kappa\}$ and
$[X]^{<\kappa}=\{Y\subseteq X:|Y|<\kappa\}$. A family $\mathcal F$ of
sets is \emph{centred} if $\bigcap\mathcal A\neq\emptyset$ for every
$\mathcal A\in[\mathcal F]^{<\aleph_0}\setminus\{\emptyset\}$.
The set of all functions from a set $A$ to a set $B$ is denoted by
$\mbox{}^AB$; if $\alpha$ is an ordinal, then we write
$\mbox{}^{<\alpha}B=\bigcup_{\xi\in\alpha}\mbox{}^\xi B$ and
$\mbox{}^{\le\alpha}B=\mbox{}^{<\alpha+1}B$.
For $f,g\in\mbox{}^AB$, the notation $f=^*g$ means that the set
$\{a\in A:f(a)\neq g(a)\}$ is finite.
If $A$ and $B$ are non-empty sets, we write
$Fn(A,B)=\bigcup\{\mbox{}^FB:F\in[A]^{<\aleph_0}\}$.
The two-point set
$\{0,1\}$ will be denoted simply by $2$, and will always be
regarded as a discrete space.
If $X=\prod_{i\in I}X_i$ is a product, then for each $j\in I$ the
function $\pi_j:X\rightarrow X_j$ is the projection $\pi_j((x_i)_{i\in
I})=x_j$.

If $X$ is a topological space and $p\in X$, we say that $p$ is a
$G_\delta$ point of $X$ if the set $\{p\}$ is a countable intersection
of open subsets of $X$. The \emph{weight} of a space $X$ is
$w(X)=\min\{|\mathcal B|:\mathcal B$ is a base for $X\}+\aleph_0$. The
\emph{pseudocharacter} of a $T_1$ space $X$ is
$\psi(X)=\sup\{\psi(x,X):x\in X\}$, where
$\psi(x,X)=\min\{\kappa:\{x\}$ is an intersection of $\kappa$ open
subsets of $X\}+\aleph_0$; for compact $T_2$ spaces,
$\psi(x,X)=\chi(x,X)=\min\{|\mathcal V|:\mathcal V$ is a local base at
$x$ in $X\}+\aleph_0$, the \emph{character} of $x$ in $X$.
If $\langle X,\mathcal{T}\rangle$ is a topological space and $x\in X$,
a \emph{$\pi$-base} for $x$ in $X$ is a family $\mathcal
P\subseteq\mathcal{T}\setminus\{\emptyset\}$ such that every neighbourhood of
$x$ in $X$ includes some element of $\mathcal P$; the
\emph{$\pi$-character} of $\langle X,\mathcal{T}\rangle$, then, is
$\pi\chi(X)=\sup\{\pi\chi(x,X):x\in X\}$, where
$\pi\chi(x,X)=\min\{|\mathcal P|:\mathcal P$ is a $\pi$-base for $x$
in $X\}+\aleph_0\le\chi(x,X)$. A topological space $X$ is said to
include a copy of a topological space $Y$ if it has a subspace
homeomorphic to $Y$.
We say that a topological space $X$ is \emph{scattered} if every
non-empty $Y\subseteq X$ has a point that is isolated in $Y$.
A topological space is \emph{extremally disconnected} if
the closure of every open subset is open.

A \emph{tree} is a strict partial order $\langle T,\le\rangle$ (often
written simply $T$) such that, for each $t\in T$, the set
$T^\downarrow(t)=\{t'\in T:t'<t\}$ is well-ordered. For each $t\in T$,
the \emph{height} of $t$ in $T$, denoted by $\mathrm{ht}_T(t)$, is the
order type of $T^\downarrow(t)$. If $\alpha$ is an
ordinal, the $\alpha$-th \emph{level} of $T$ is $T_\alpha=\{t\in
T:\mathrm{ht}_T(t)=\alpha\}$; $T$ is \emph{rooted} if $|T_0|=1$. The
\emph{height of $T$} is the least ordinal $\eta$ with
$T_\eta=\emptyset$. We say that $T$ is \emph{Hausdorff} if, whenever
$\alpha$ is a limit ordinal and $t,t'\in T_\alpha$ are distinct, the
sets $T^\downarrow(t)$ and $T^\downarrow(t')$ are distinct.
A \emph{subtree} of $\langle T,\le\rangle$ is any subset $T'\subseteq
T$ considered with the restriction of $\le$ to $T'\times T'$ (which
makes $T'$ a tree as well); if a subtree $T'$ is downwards closed,
i.e. $T^\downarrow(t)\subseteq T'$ for all $t\in T'$, then $T'$
is said to be an \emph{initial part} of $T$.

A \emph{chain} in $T$ is any subset of $T$
that is linearly ordered by $\le$. A \emph{branch} is a chain
that is maximal with respect to inclusion of sets ($\subseteq$). The
\emph{cofinality} of a branch $B$
is the least cardinality of a \emph{cofinal} subset of $B$,
i.e. a subset $C\subseteq B$ such that $\forall t\in B\;\exists
t'\in C\;(t\le t')$; equivalently, the cofinality of $B$ is
$\mathrm{cf}(\beta)$, where $\beta$ is the order type
of $B$ with the order induced by $\le$. A branch $B$ is \emph{cofinal}
in $T$ if $B\cap T_\alpha\neq\emptyset$ for all $\alpha$ with
$T_\alpha\neq\emptyset$.

Suppose $\langle T,\le\rangle$ is a Hausdorff tree where, for each
$t\in T$, the set $S(t)$ of all the immediate successors of $t$ in
$\le$ is linearly ordered by a relation $\prec_t$. The
\emph{lexicogra\-phical ordering} on the set of all the branches of
$T$ is the linear order $\prec$ defined by $B\prec B'\leftrightarrow
v\prec_t v'$, where $t$ is the $\le$-greatest element of $B\cap B'$
and $v,v'\in S(t)$ are such that $v\in B$ and $v'\in B'$.

If $\kappa$ is a cardinal, we say that $T$ is a \emph{$\kappa$-tree}
if $T$ has height $\kappa$ and all the levels of $T$ have size
$<\kappa$. A \emph{$\kappa$-Aronszajn tree} is a $\kappa$-tree with no
cofinal branch.
A \emph{Kurepa tree} is an $\omega_1$-tree with at least $\aleph_2$
cofinal branches. \emph{Kurepa's Hypothesis} ($\mathrm{KH}$) is the
statement ``there is a Kurepa tree''.

An uncountable cardinal $\kappa$ is \emph{(strongly) inaccessible} if it
is regular and $2^\lambda<\kappa$ for every cardinal
$\lambda<\kappa$.
If $\kappa$ is inaccessible and there is no
$\kappa$-Aronszajn tree, then $\kappa$ is said to be \emph{weakly
compact}.
It is known by a result of R. Solovay (see
\cite[Theorem 8.11]{stevoh}) that, if $\omega_2$ is not 
inaccessible in the constructible universe $\mathbf L$, then
$\mathrm{KH}$ holds.
Whenever we say that two statements are equiconsistent, we mean that
they are equiconsistent relative to $\mathrm{ZFC}$.

\section{Indestructibility \emph{versus}
  $\mathsf{S}_1^{\omega_1}(\mathcal{O},\mathcal{O})$}

\label{sectionlex}

The following terminology is taken from
\cite{kunentree}:

\begin{defin}
\label{kunen}

Let $\kappa$ be an infinite cardinal. A \emph{$\kappa$-\v
Cech-Posp\'i\v sil tree} in a topolo\-gical space $X$ is an indexed
family $\langle F_s:s\in\mbox{}^{\le\kappa}2\rangle$ satisfying:
\begin{itemize}
\item[$(i)$]
each $F_s$ is a non-empty closed subset of $X$;
\item[$(ii)$]
$F_s\supseteq F_t$ whenever $s\subseteq t$;
\item[$(iii)$]
$F_{s^\smallfrown(0)}\cap F_{s^\smallfrown(1)}=\emptyset$;
\item[$(iv)$]
if $\gamma\le\kappa$ is a non-zero limit ordinal and
$s\in\mbox{}^\gamma 2$, then
$F_s=\bigcap_{\alpha\in\gamma}F_{s\upharpoonright\alpha}$.
\end{itemize}
\end{defin}

\begin{prop}
\label{cptree}

If there is an $\omega_1$-\v Cech-Posp\'i\v sil tree in a
topological space $X$, then player One has a winning strategy in the
game $\mathsf{G}_1^{\omega_1}(\mathcal{O}_X,\mathcal{O}_X)$.

\end{prop}

\begin{proof}

Let $\langle F_s:s\in\mbox{}^{\le\omega_1}2\rangle$ be an
$\omega_1$-\v Cech-Posp\'i\v sil tree in $X$. For each
$s\in\mbox{}^{\le\omega_1}2$, consider $U_s=X\setminus F_s$. This
defines the following strategy for player One in
$\mathsf{G}_1^{\omega_1}(\mathcal{O}_X,\mathcal{O}_X)$: One starts the
play with the open cover $\{U_{(0)},U_{(1)}\}$; if $\alpha\in\omega_1$
and $s\in\mbox{}^{\alpha}2$ are such that, for each $\beta\in\alpha$,
Two's move in the $\beta$-th inning of this play was
$U_{s\upharpoonright(\beta+1)}$,
then One's move in the $\alpha$-th inning is
$\{U_{s^\smallfrown(0)},U_{s^\smallfrown(1)}\}$. When the play ends,
Two will have played the sets
$(U_{t\upharpoonright(\beta+1)})_{\beta\in\omega_1}$
for some $t\in\mbox{}^{\omega_1}2$; thus One wins since
$\bigcup\{U_{t\upharpoonright(\beta+1)}:\beta\in\omega_1\}=U_t\neq X$.
\end{proof}

\begin{corol}
\label{compactT2}

Every compact Hausdorff space with no $G_\delta$ points is
destructible.

\end{corol}

\begin{proof}

It is a famous theorem \cite{cechp} that, if $X$ is a compact
Hausdorff space with no $G_\delta$ points, then there is an
$\omega_1$-\v Cech-Posp\'i\v sil tree in $X$. The result
then follows from Theorem \ref{st1} and Proposition \ref{cptree}.
\end{proof}

\begin{corol}
\label{convseq}

If an infinite compact Hausdorff space $X$ does not contain non-trivial
convergent sequences, then $X$ is destructible.

\end{corol}

In the proof of Corollary \ref{convseq}, we shall make use of the
following (probably folklore) fact:

\begin{lemma}
\label{scatseq}

If an infinite compact Hausdorff space $X$ is scattered, then there is
a non-trivial convergent sequence in $X$.

\end{lemma}

\begin{proof}

Let $I\subseteq X$ be the set of isolated points of $X$. As $X$
is compact Hausdorff, we cannot have $I=X$, for otherwise $X$ would be
finite. Thus $X\setminus I\neq\emptyset$. Since $X$ is scattered,
there are $x\in X\setminus I$ and a neighbourhood $V$ of $x$
such that $V\cap(X\setminus I)=\{x\}$, i.e. $V\subseteq
I\cup\{x\}$. By regularity of $X$, we may assume that $V$ is closed,
hence compact. Since $x\notin I$, we have that $x$ is
not isolated in $V$; in particular, $V$ is
infinite. It follows that $V$ is homeomorphic to the one-point
compactification of the infinite discrete set $V\cap I$; therefore,
there is a non-trivial convergent sequence in $V$, and hence in $X$.
\end{proof}

\begin{proof}[Proof of Corollary \ref{convseq}]

Let $X=S\dot\cup P$ be the Cantor-Bendixson decomposition of $X$,
where $S$ is an open scattered subspace of $X$ and $P$ is a closed
subspace of $X$ no point of which is isolated (in $P$). If we had
$P=\emptyset$, then $X$ would be compact scattered $T_2$, and hence
would have a non-trivial convergent sequence by Lemma \ref{scatseq};
therefore, $P\neq\emptyset$. As no point of $P$ is isolated, it
follows that $\chi(x,P)\ge\aleph_1$ for all $x\in P$, since otherwise
there would be a non-trivial sequence converging to a point of first
countability of $P$. By Corollary \ref{compactT2}, $P$ is
destructible; therefore, $X$ is destructible since indestructibility
is hereditary with respect to closed subspaces --- see Lemma
\ref{preserv}.
\end{proof}

We now turn to the proof of Theorem \ref{lex}.

For any distinct $f,g\in\mbox{}^{\omega_1}2$, define
$\Delta(f,g)=\min\{\xi\in\omega_1:f(\xi)\neq g(\xi)\}.$
Let $\prec$ be the lexicographical ordering on the set
$\mbox{}^{\omega_1}2$, i.e., for any distinct
$f,g\in\mbox{}^{\omega_1}2$ we have that $f\prec g$ if and only if
$f(\Delta(f,g))=0$ and $g(\Delta(f,g))=1$.
Finally, let $X$ be the
linearly ordered topological space $\mbox{}^{\omega_1}2$ obtained from
the ordering $\prec$.

The following fact is well-known (see e.g. \cite[Lemma
13.17]{gil}):

\begin{lemma}
\label{supinf}

Every non-empty subset of $X$ has a least upper
bound and a greatest lower bound.

\end{lemma}

As $X$ is a linearly ordered space, Lemma \ref{supinf} can be restated
as:

\begin{lemma}
\label{compact}

$X$ is compact.

\end{lemma}

Next we will show that $X$ satisfies the hypotheses of Corollary
\ref{compactT2}. In order to do so, we shall make use of the following
two lemmas.

\begin{lemma}
\label{sucpred}

Let $f,g\in\mbox{}^{\omega_1}2$ be such that $f\prec g$. Then
the open interval $(f,g)$ is empty if and only if for every
$\xi\in\omega_1$ with $\xi>\Delta(f,g)$ we have $f(\xi)=1$ and
$g(\xi)=0$.

\end{lemma}

\begin{proof}
See 13.16 in \cite{gil}.
\end{proof}

\begin{lemma}
\label{limit}

Let $f\in\mbox{}^{\omega_1}2$.

\begin{itemize}

\item[$(a)$]

If $f=\sup A$, where $A\subseteq\mbox{}^{\omega_1}2$ is countable
and has no greatest element, 
then the set $\{\xi\in\omega_1:f(\xi)=1\}$ is
countable.

\item[$(b)$]

If $f=\inf A$, where $A\subseteq\mbox{}^{\omega_1}2$ is countable
and has no least element, 
then the set $\{\xi\in\omega_1:f(\xi)=0\}$ is
countable.

\end{itemize}

\end{lemma}

\begin{proof}

We prove $(a)$; $(b)$ is analogous.

Let $A\subseteq\mbox{}^{\omega_1}2$ be countable with no greatest
element and such that $f=\sup A$. Since $A$ has no greatest element,
we have that $f\notin A$;
let then
$\delta=\sup\{\Delta(x,f):x\in A\}+1\in\omega_1$. We claim that
$f(\xi)=0$ for all $\xi\in\omega_1\setminus\delta$.
Indeed, if $\alpha\in\omega_1\setminus\delta$ were such that
$f(\alpha)=1$, we would have that the function
$g\in\mbox{}^{\omega_1}2$ defined by
$$
g(\xi)=\left\{
\begin{array}{lll}0&&\textrm{
    if }\xi=\alpha\\
f(\xi)&&\textrm{ otherwise}
\end{array}\right.
$$
would be an upper bound for $A$ satisfying
$g\prec f$, thus contradicting the assumption that $f=\sup A$.
\end{proof}

\begin{lemma}
\label{noGd}

No point of $X$ is a $G_\delta$.

\end{lemma}

\begin{proof}

Suppose that $f\in X$ is a $G_\delta$ point.
Let us assume, for a moment, that $f$ is not constant.
Then there are sequences $(x_n)_{n\in\omega}$ and $(y_n)_{n\in\omega}$
in $X$ such that $\{g\in X:x_n\prec g\prec y_n$ for all
$n\in\omega\}=\{f\}$. Let $x=\sup\{x_n:n\in\omega\}$ and
$y=\inf\{y_n:n\in\omega\}$ --- note that $x$ and $y$ are well-defined
by virtue of Lemma \ref{supinf}. Since $x_n\prec f\prec y_n$ for all
$n\in\omega$, we must have $x\preceq f\preceq y$.

Note that $(\dagger)$ $x\prec f$ implies $(x,f)=\emptyset$ and that
$(\ddagger)$ $f\prec y$ implies $(f,y)=\emptyset$. Hence, in
light of Lemma \ref{sucpred}, $x\prec f\prec y$ cannot hold. We cannot
have $x=f=y$ either, in view of Lemma \ref{limit}. We are thus left
with the cases $x\prec f=y$ and $x=f\prec y$, which are also seen to
be impossible by putting Lemmas \ref{sucpred} and \ref{limit}
together with $(\dagger)$ and $(\ddagger)$.

Finally, we note that essentially the same argument applies to dealing
with the cases $f\equiv 0$ and $f\equiv 1$.
\end{proof}

Lemmas \ref{compact} and \ref{noGd}, together with Corollary
\ref{compactT2}, yield:

\begin{corol}
\label{indestr}

$X$ is destructible.

\end{corol}

We shall now see that $X$ satisfies
$\mathsf{S}_1^{\omega_1}(\mathcal{O},\mathcal{O})$ under the Continuum
Hypothesis.

\begin{lemma}
\label{aux}

If $h\in X$ is an accumulation point of a subset $A\subseteq X$,
then $h=\sup\{f\in A:f\prec h\}$ or $h=\inf\{f\in A:h\prec f\}$.

\end{lemma}

\begin{proof}

Let $L=\{f\in A:f\prec h\}$ and $R=\{f\in A:h\prec f\}$. We first deal
with the case where $L$ and $R$ are both non-empty.

Let $x=\sup L$ and $y=\inf R$. Clearly, $x\preceq h\preceq y$. If
we had $x\prec h\prec y$, it would follow that $(x,y)\cap
A\subseteq\{h\}$, thus contradicting the assumption that $h$ is an
accumulation point of $A$. Therefore, $h=x$ or $h=y$.

We can proceed similarly in the cases $L=\emptyset$ and $R=\emptyset$.
\end{proof}

\begin{lemma}
\label{conc}

If a closed subset of $F\subseteq X$ is infinite, then some $h\in F$
is eventually constant.

\end{lemma}

\begin{proof}

Pick an arbitrary $A\in[F]^{\aleph_0}$. As $F$ is compact, there is a
point $h\in F$ that is an accumulation point of $A$. We can apply
Lemma \ref{aux} and assume, without loss of generality, that
$h=\sup\{f\in A:f\prec h\}$. Since $h\notin\{f\in A:f\prec h\}$, it
follows from Lemma \ref{limit} that $\{\xi\in\omega_1:h(\xi)=1\}$ is
countable.
\end{proof}

\begin{prop}
\label{CH}

$\mathrm{CH}$ implies
$\mathsf{S}_1^{\omega_1}(\mathcal{O}_X,\mathcal{O}_X)$.

\end{prop}

\begin{proof}

Let
$C=\{f\in\mbox{}^{\omega_1}2:f$ is eventually constant$\}$. Using
$\mathrm{CH}$, write
$C=\{f_\alpha:\alpha\in\omega_1\setminus\omega\}$.

Now let $(\mathcal{U}_\alpha)_{\alpha\in\omega_1}$ be a sequence of
open covers of $X$. For each $\alpha\in\omega_1\setminus\omega$, pick
$U_\alpha\in\mathcal{U}_\alpha$ such that $f_\alpha\in U_\alpha$. By
Lemma \ref{conc}, the set
$F=X\setminus\bigcup\{U_\alpha:\alpha\in\omega_1\setminus\omega\}$
is finite. Thus we can cover all the points in $F$ by choosing one
open set $U_n$ in each $\mathcal{U}_n$ with $n\in\omega$, and so we
get $(U_\alpha)_{\alpha\in\omega_1}$ such that
$U_\alpha\in\mathcal{U}_\alpha$ for all $\alpha\in\omega_1$ and
$X=\bigcup\{U_\alpha:\alpha\in\omega_1\}$.
\end{proof}

Note that, in order to conclude $\mathsf{S}_1^{\omega_1}(\mathcal
O_X,\mathcal O_X)$ in Proposition \ref{CH}, it would suffice to find a
subspace $Y\subseteq X$ satisfying $\mathsf S_1^{\omega_1}(\mathcal
O_Y,\mathcal O_Y)$ and such that $\mathsf S_1^{\omega_1}(\mathcal
O_{X\setminus U},\mathcal O_{X\setminus U})$ for every open set $U$
with $Y\subseteq U\subseteq X$. The fact that we could get
considerably more than this in our example (which is due to Lemma
\ref{conc}) suggests that there might still be room for improvement.

\begin{prob}
\label{CH?}

Is there such a counterexample in $\mathrm{ZFC}$? What about under
weaker additional hypotheses, e.g. $\mathrm{MA}$?

\end{prob}

Finally, let us make an observation about this space of which we shall
make use in the next section:

\begin{lemma}
\label{pichi}

If $F\subseteq X$ is closed, then there is $h\in F$ such that
$\pi\chi(h,F)=\aleph_0$.

\end{lemma}

\begin{proof}

If $F$ is finite, the result is immediate. If $F$ is infinite, pick
an arbitrary $A\in[F]^{\aleph_0}$; since $F$ is compact, $A$ has an
accumulation point $h\in F$. As in the proof of Lemma \ref{conc}, we
can apply Lemma \ref{aux} and assume that
$h=\sup\{f\in A:f\prec h\}$. It follows that
$\{(f,h)\cap\, F:f\in A,f\prec h\}$ is a countable $\pi$-base
for $h$ in $F$.
\end{proof}

\section{Some classes of destructible spaces}

Recall that a Hausdorff space is \emph{dyadic} if it is a continuous
image of the Cantor cube $2^\kappa$ for some infinite cardinal
$\kappa$.

\begin{thm}[\v Sanin \cite{sanin}]
\label{sanin}

If a dyadic space has weight $\lambda$, then it is a continuous image
of $2^\lambda$.

\end{thm}

In what follows we shall make use of the following fact, the
straightforward proof of which we omit:

\begin{lemma}
\label{preserv}

If a topological space $X$ satisfies
$\mathsf{S}_1^{\omega_1}(\mathcal{O},\mathcal{O})$ (respectively, One
has no winning strategy in $\mathsf{G}_1^{\omega_1}(\mathcal{O},\mathcal{O})$),
then every closed subspace of $X$ and every continuous image of $X$
also satisfy $\mathsf{S}_1^{\omega_1}(\mathcal{O},\mathcal{O})$
(respectively, One has no winning strategy in
$\mathsf{G}_1^{\omega_1}(\mathcal{O},\mathcal{O})$).

\end{lemma}

\begin{lemma}
\label{2^k}

The following conditions are equivalent for an infinite cardinal
$\kappa$:

\begin{itemize}

\item[$(a)$]
$2^\kappa$ is indestructible;

\item[$(b)$]
$\mathsf{S}_1^{\omega_1}(\mathcal{O}_{2^\kappa},\mathcal{O}_{2^\kappa})$;

\item[$(c)$]
$\kappa=\omega$.

\end{itemize}

\end{lemma}

\begin{proof}

We already know that $(a)\rightarrow(b)$ by Theorem \ref{st1}.

For $(b)\rightarrow(c)$, suppose to the contrary that $\kappa$ is
uncountable. Then $2^\kappa$ includes a copy of
$2^{\omega_1}$, which is closed since $2^{\omega_1}$ is compact. Thus,
by Lemma \ref{preserv},
$\mathsf{S}_1^{\omega_1}(\mathcal{O},\mathcal{O})$ fails for
$2^\kappa$ since it fails for $2^{\omega_1}$: in order to see this, we
need only consider for each $\alpha\in\omega_1$ the open cover
$\mathcal{U}_\alpha=\{\pi_\alpha^{-1}[\{0\}],\pi_\alpha^{-1}[\{1\}]\}$
of $2^{\omega_1}$, and no matter how $U_\alpha\in\mathcal{U}_\alpha$
are chosen for $\alpha\in\omega_1$ the set
$2^{\omega_1}\setminus\bigcup\{U_\alpha:\alpha\in\omega_1\}$ will
consist of exactly one point.

$(c)\rightarrow(a)$ is a direct consequence of Theorem 6(a) of
\cite{tall95}, which states that every hereditarily Lindel\"of space
is indestructible.
\end{proof}

We can now characterize the class of dyadic indestructible spaces:

\begin{corol}
\label{dyadic}

The following are equivalent for a dyadic space $X$:

\begin{itemize}

\item[$(a)$]
$X$ is indestructible;

\item[$(b)$]
$\mathsf{S}_1^{\omega_1}(\mathcal{O}_X,\mathcal{O}_X)$;

\item[$(c)$]
$X$ does not include a copy of $2^{\omega_1}$;

\item[$(d)$]
$w(X)=\aleph_0$.

\end{itemize}

\end{corol}

\begin{proof}

Again, $(a)\rightarrow(b)$ is clear in view of Theorem \ref{st1}, and
$(b)\rightarrow(c)$ was already observed in the proof of Lemma
\ref{2^k}.
The implication $(c)\rightarrow(d)$ is a particular case of
J. Hagler's Theorem 1 of \cite{hagler}.
Finally, Lemmas \ref{preserv} and \ref{2^k} together with Theorem
\ref{sanin} yield $(d)\rightarrow(a)$. 
\end{proof}

This alternative proof of $(a)\rightarrow(d)$ in Corollary
\ref{dyadic} may also be of interest:

\begin{proof}[Proof of $(a)\rightarrow(d)$]

It was proven by B. Efimov in \cite{efimov63a}
that, if a dyadic space $X$ has uncountable weight, then the set of
$G_\delta$ points of $X$ is not dense in $X$. Let then $\Omega$ be a
non-empty open subset of $X$ with $\psi(p,X)>\aleph_0$ for every
$p\in\Omega$; we can then use regularity of $X$ to find a non-empty
$G_\delta$ subset $F$ of $\Omega$ such that $F$ is closed in
$X$. Since $F$ is a $G_\delta$ in $X$ as well, it follows that
$\psi(p,F)=\psi(p,X)>\aleph_0$ for all $p\in F$. By Corollary
\ref{compactT2}, $F$ is destructible; therefore, $X$ is destructible
by Lemma \ref{preserv} and Theorem \ref{st1}.
\end{proof}

\begin{corol}
\label{dyadicsize}

Every dyadic space of cardinality greater than $\mathfrak{c}$ is
destructible.

\end{corol}

\begin{proof}

By Lemma \ref{sanin} and Corollary \ref{dyadic}, we have that an
indestructible dyadic space $X$ must be a continuous image of the
Cantor space $2^\omega$, and thus $|X|\le 2^{\aleph_0}=\mathfrak{c}$.
\end{proof}

Corollary \ref{dyadic} tells us that the properties of
including a copy of $2^{\omega_1}$ and being destructible are
equivalent for dyadic spaces. We shall now try to draw some lines
between them for compact Hausdorff spaces in general.

Recall the following theorem of B. E. \v Sapirovski\u\i\mbox{}
\cite{sap} (see also 3.18 in \cite{juhasz10}):

\begin{thm}[\v Sapirovski\u\i\mbox{} \cite{sap}]
\label{sap}

The following conditions are equivalent for a compact $T_2$ space $X$
and an uncountable cardinal $\kappa$:

\begin{itemize}

\item[$(a)$]
$[0,1]^\kappa$ is a continuous image of $X$;

\item[$(b)$]
$2^\kappa$ is a continuous image of a closed subset of $X$;

\item[$(c)$]
there is a closed non-empty $F\subseteq X$ such that
$\pi\chi(x,F)\ge\kappa$ for all $x\in F$;

\item[$(d)$]
there is a $\kappa$-dyadic system in $X$, i.e.
an indexed family $\langle F_\alpha^i:\alpha\in\kappa, i\in 2\rangle$
of closed subsets of $X$ such that $F_\alpha^0\cap
F_\alpha^1=\emptyset$ for all $\alpha\in\kappa$ and
$\bigcap\{F_\xi^{p(\xi)}:\xi\in\mathrm{dom}(p)\}\neq\emptyset$ for all
$p\in Fn(\kappa,2)$.

\end{itemize}

\end{thm}

The case $\kappa=\omega_1$ of Theorem \ref{sap} is of interest for the
study of destructibility of compact spaces:

\begin{lemma}
\label{diagram-dyadic}

Let $X$ be a compact
space.

\begin{itemize}

\item[$(a)$]

If $X$ includes a copy of $2^{\omega_1}$, then there is an
$\omega_1$-dyadic system in $X$.

\item[$(b)$]

If there is an $\omega_1$-dyadic system in $X$, then
$\neg\mathsf{S}_1^{\omega_1}(\mathcal O_X,\mathcal O_X)$.

\end{itemize}

\end{lemma}

\begin{proof}

Since $\langle\pi_\alpha^{-1}[\{i\}]:\alpha\in\omega_1,i\in
2\rangle$ is an $\omega_1$-dyadic system in $2^{\omega_1}$, it
defines an $\omega_1$-dyadic system in $X$. This proves $(a)$.

For $(b)$, let $\langle F_\alpha^i:\alpha\in\omega_1, i\in 2\rangle$
be an $\omega_1$-dyadic system in $X$. For each
$\alpha\in\omega_1$, let $\mathcal{U}_\alpha=\{X\setminus
F_\alpha^0,X\setminus F_\alpha^1\}$. The sequence $(\mathcal
U_\alpha)_{\alpha\in\omega_1}$ witnesses the failure of $\mathsf
S_1^{\omega_1}(\mathcal O_X,\mathcal O_X)$: for every
$f\in\mbox{}^{\omega_1}2$, we have that
$\bigcap\{F_\alpha^{f(\alpha)}:\alpha\in\omega_1\}\neq\emptyset$ since
this is the intersection of a centred family of closed subsets of a
compact space.
\end{proof}

Let us also evoke Theorem 3.1 of \cite{pierce} and Corollary 4 of
\cite{balcar}:

\begin{thm}[Pierce \cite{pierce}]
\label{pierce}

If $X$ is an extremally disconnected compact $T_2$ space, then
$w(X)^{\aleph_0}=w(X)$.

\end{thm}

\begin{thm}[Balcar-Franek \cite{balcar}]
\label{kisl}

If $X$ is an extremally disconnected compact $T_2$ space, then $X$ can
be continuously mapped onto the Cantor cube $2^{w(X)}$.

\end{thm}

In particular, if $X$ is an extremally disconnected compact $T_2$
space of uncountable weight, it follows from Theorem \ref{pierce} that
$w(X)\ge\mathfrak c$; thus, by Theorems \ref{sap} and
\ref{kisl}, there is a $\mathfrak c$-dyadic system in $X$ ---
which, in view of Lemma \ref{diagram-dyadic} and Theorem \ref{st1},
implies that $X$ is destructible.\footnote{\label{foot}This fact also
  follows from Corollary \ref{convseq}, since extremally disconnected
  $T_2$ spaces have no non-trivial convergent sequences (see
  \cite[Theorem 1.3]{gleason}).} This shows that perfect
preimages of compact Rothberger $T_2$ spaces need not be Rothberger or
even indestructible: let $X$ be the one-point compactification of a
discrete space of cardinality $\aleph_1$; then the Stone space
$\theta(X)$ of the regular closed algebra of $X$ is a perfect preimage
of $X$ (see \cite[Theorem 3.2]{gleason}), yet $\theta(X)$ is
destructible --- by the previous argument --- whereas $X$ is
Rothberger.

Thus we have the following diagram for $X$ compact $T_2$:\footnote{It
is worth remarking that some of the implications do not require
compactness of $X$.}

\newpage

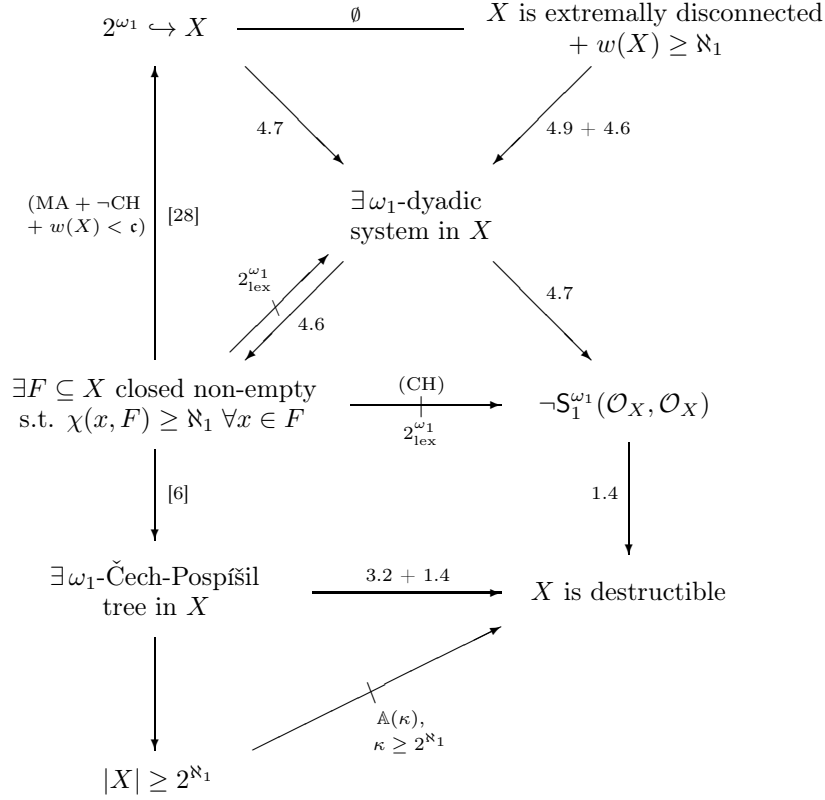
\setlength{\unitlength}{1mm}
\begin{picture}(0,130)(-10,-15)

\put(12,100){$2^{\omega_1}\hookrightarrow X$}

\put(63,102){$X$ is extremally disconnected}
\put(74,98){$+\;w(X)\ge\aleph_1$}

\put(45,77){$\exists\,\omega_1$-dyadic}
\put(45,73){system in $X$}

\put(0,52){$\exists F\subseteq X$ closed non-empty}
\put(1,48){s.t. $\chi(x,F)\ge\aleph_1\;\forall x\in F$}

\put(70,50){$\neg\mathsf S_1^{\omega_1}(\mathcal O_X,\mathcal O_X)$}

\put(5,27){$\exists\,\omega_1$-\v Cech-Posp\'i\v sil}
\put(12,23){tree in $X$}

\put(69,25){$X$ is destructible}

\put(12,0){$|X|\ge 2^{\aleph_1}$}

\put(82,46){\vector(0,-1){15}}

\put(19,45){\vector(0,-1){12}}

\put(19,20){\vector(0,-1){15}}

\put(19,57){\vector(0,1){39}}

\put(40,26){\vector(1,0){25}}

\put(45,51){\vector(1,0){20}}
\put(54,50){\footnotesize{$|$}}

\put(44,70){\vector(-1,-1){13}}

\put(64,70){\vector(1,-1){13}}

\put(31,96){\vector(1,-1){13}}

\put(77,96){\vector(-1,-1){13}}

\put(29,58){\vector(1,1){13}}

\put(60,101){\line(-1,0){30}}

\put(45,102){\scriptsize{$\emptyset$}}

\put(32,5){\vector(2,1){33}}
\put(47,12){\footnotesize{$\setminus$}}

\put(34.4,64){\footnotesize{$_\setminus$}}

\put(51.3,53){\scriptsize{$(\mathrm{CH})$}}

\put(2,77){\scriptsize{$(\mathrm{MA}+\neg\mathrm{CH}$}}
\put(2,74){\scriptsize{$+\;w(X)<\mathfrak{c})$}}

\put(20.6,75.5){\scriptsize{\cite{shapiro}}}

\put(52,46.5){\scriptsize{$2^{\omega_1}_{\mathrm{lex}}$}}

\put(30,67){\scriptsize{$2^{\omega_1}_{\mathrm{lex}}$}}

\put(32.5,87){\scriptsize{\ref{diagram-dyadic}}}

\put(71,65){\scriptsize{\ref{diagram-dyadic}}}

\put(38,61){\scriptsize{\ref{sap}}}

\put(71,87){\scriptsize{\ref{kisl} + \ref{sap}}}

\put(77,38.5){\scriptsize{\ref{st1}}}

\put(20.6,38.5){\scriptsize{\cite{cechp}}}

\put(47,27.6){\scriptsize{\ref{cptree} + \ref{st1}}}

\put(48.5,8.5){\scriptsize{$\mathbb A(\kappa)$,}}
\put(48,5){\scriptsize{$\kappa\ge 2^{\aleph_1}$}}

\end{picture}

The implication in the diagram that assumes
$\mathrm{MA}+\neg\mathrm{CH}+w(X)<\mathfrak{c}$
is a parti\-cular case of L. B. Shapiro's Theorem 1.3 of
\cite{shapiro}.
Each non-implication is marked with the space that is a counterexample
to it. We denote by $\mathbb A(\kappa)$ the one-point compactification
of the discrete space of size $\kappa$, and by
$2^{\omega_1}_{\mathrm{lex}}$ the space considered in the proof of
Theorem \ref{lex} presented in Section 3.

The fact that $2^{\omega_1}_{\mathrm{lex}}$ shows the
non-implication that points to ``there is an $\omega_1$-dyadic system
in $X$'' in the diagram follows from Lemma \ref{pichi} and Theorem
\ref{sap}; since this space has weight $\mathfrak c$ (see the
paragraph preceding Lemma \ref{compl}), we have, in particular, that
the weight constraint in Shapiro's theorem is the best possible.
This can also be seen from the fact that the properties
``$2^{\omega_1}\hookrightarrow X$'' and ``$X$ is extremally
disconnected'' cannot hold simultaneously --- represented in
the diagram by the line labeled with the symbol $\emptyset$ ---, which
follows from the (already previously mentioned) fact that every
convergent sequence in an extremally disconnected Hausdorff space is
trivial; thus any extremally disconnected compact Hausdorff space of
uncountable weight --- e.g. $\beta\omega$ --- has an
$\omega_1$-dyadic system but does not include a copy of
$2^{\omega_1}$.
We point out again that Theorem \ref{pierce} implies that
any such space would necessarily have weight $\ge\mathfrak c$,
which is in accordance with Shapiro's result.

From such considerations regarding the space from Section 3 we can
derive the following observation concerning Shapiro's theorem:

\begin{prop}
\label{shapw1}

Let $(\Upsilon)$ denote the statement:

\begin{quote}

If $X$ is compact $T_2$ of weight $\aleph_1$ and there is a closed
non-empty $F\subseteq X$ with $\chi(x,F)=\aleph_1$ for all $x\in F$,
then $X$ includes a copy of $2^{\omega_1}$.

\end{quote}

Then
$\mathrm{MA}+\neg\mathrm{CH}\rightarrow(\Upsilon)\rightarrow\neg\mathrm{CH}$.

\end{prop}

Still in the diagram, it is also worth noting that the existence of an
$\omega_1$-dyadic system does not follow from the existence of an
$\omega_1$-\v Cech-Posp\'i\v sil tree.
Note that there is an $\omega_1$-dyadic system in $X$ if and only if
there is a sequence $(\mathcal{U}_\alpha)_{\alpha\in\omega_1}$ of
two-element open covers of $X$ witnessing the failure of $\mathsf
S_1^{\omega_1}(\mathcal O_X,\mathcal O_X)$. Similarly, there is an
$\omega_1$-\v Cech-Posp\'i\v sil tree in $X$ if and only if One has a
winning strategy in $\mathsf G_1^{\omega_1}(\mathcal O_X,\mathcal
O_X)$ in which she plays only two-element open covers of $X$. This
raises the following questions:

\begin{prob}
\label{S1bin}

For $X$ compact, is $\mathsf S_1^{\omega_1}(\mathcal O_X,\mathcal
O_X)$ equivalent to the non-existence of an $\omega_1$-dyadic system
in $X$?

\end{prob}

\begin{prob}
\label{bindestr}

Is destructibility of a space $X$ equivalent to the existence of an
$\omega_1$-\v Cech-Posp\'i\v sil tree in $X$?

\end{prob}

\section{Indestructibility and large cardinals}

\label{sectiontrees}

In the paper \cite{usuba}, the following results were established:

\begin{thm}[Tall-Usuba \cite{usuba}]
\label{BC}

If it is consistent with $\mathrm{ZFC}$ that there is an
inaccessible cardinal, then it is consistent with $\mathrm{ZFC}$ that
every Lindel\" of $T_3$ indestructible space
of weight $\le\aleph_1$ has size $\le\aleph_1$.

\end{thm}

\begin{thm}[Tall-Usuba \cite{usuba}]
\label{proj}

If it is consistent with $\mathrm{ZFC}$ that there is an
inaccessible cardinal, then it is consistent with $\mathrm{ZFC}$ that
the $\aleph_1$-Borel Conjecture holds.

\end{thm}

\begin{thm}[Tall-Usuba \cite{usuba}]
\label{psi}

If it is consistent with $\mathrm{ZFC}$ that there is a weakly compact
cardinal, then it is consistent with $\mathrm{ZFC}$ that
there is no Lindel\" of $T_1$ space of pseudocharacter
$\le\aleph_1$ and size $\aleph_2$.

\end{thm}

We will show that the large cardinal assumptions are necessary in
the above theorems
by proving:

\begin{thm}
\label{inaccBC}

If
$\mathrm{KH}$ holds, then there is a compact $T_2$
indestructible space of weight $\le\aleph_1$ and size $>\aleph_1$.

\end{thm}

\begin{thm}
\label{weak}

If $\omega_2$ is not weakly compact in
$\mathbf L$, then there is a Lindel\"of $T_3$ indestructible space of
pseudocharacter $\le\aleph_1$ and size $\aleph_2$.

\end{thm}

Both of these results were obtained independently, around the same
time as we did, by T. Usuba. We are very thankful to
Stevo Todor\v cevi\'c for bringing to our attention that the
construction found in \cite{stevoh} would be useful for
the topic we discuss in this section.

We point out that indestructibility in Theorem \ref{weak} is relevant
in view of the fact that the existence of an indestructible
counterexample to the conclusion of Theorem \ref{psi} is shown in
\cite{usuba} by adding $\aleph_3$ Cohen subsets of $\omega_1$ under
the hypothesis that $\omega_2$ is not weakly compact in $\mathbf{L}$.

The following construction is taken from Section 8 of Todor\v
cevi\'c's article \cite{stevoh}.
Let $T$ be a rooted Hausdorff tree such that all the levels and
branches of $T$ have size $\le\aleph_1$.
We may assume that $T$ is an initial part of
$\langle\mbox{}^{<\omega_2}(\omega_1+1),\subseteq\rangle$ satisfying
the following condition:
$$
(*)\textrm{ for each }t\in T\textrm{, the set
}\{\xi\le\omega_1:t^\smallfrown(\xi)\in T\}\textrm{ is a successor ordinal.}
$$
Let $\prec$ be the ordering on the set
$L_T=\{\bigcup B:B$
is a branch of $T\}$ naturally induced by the lexicographical ordering
of the branches of $T$, and regard $L_T$ as a linearly ordered
topological space. The argument present in Lemma 8.1(i) of
\cite{stevoh} applies in this case to show that $w(L_T)\le|T|$.
For each $t\in T$, let $L_T(t)=\{s\in L_T:t\subseteq s\}$.

\begin{lemma}[see \cite{stevoh}]
\label{compl}

$L_T(t)$ is a compact subspace of $L_T$ for all $t\in T$.

\end{lemma}

\begin{proof}

We will prove that $L_T=L_T(\emptyset)$ is compact; the
general case is analogous. Since $L_T$ is a linearly ordered space,
compactness of $L_T$ is equivalent to completeness of the linear
ordering of $L_T$.

Let $A\subseteq L_T$ be arbitrary. For each $s\in L_T$, let $\tilde
s=s\cup\{\langle\xi,0\rangle:\xi\in\omega_2\setminus\mathrm{dom}(s)\}\in\mbox{}^{\omega_2}(\omega_1+1)$.
Now define $f\in\mbox{}^{\omega_2}(\omega_1+1)$ recursively by
$$
f(\alpha)=\sup\{\tilde s(\alpha):s\in A\textrm{ and
}s\upharpoonright\alpha=f\upharpoonright\alpha\}
$$
for all $\alpha\in\omega_2$. There must be some $\beta\in\omega_2$
with $f\upharpoonright\beta\notin T$, since otherwise
$\{f\upharpoonright\beta:\beta\in\omega_2\}$ would be a branch of
cardinality $\aleph_2$ in $T$. Let then $\beta_0$ be the least such
$\beta$, and define $u=f\upharpoonright\beta_0$.

\emph{Case 1.} $\beta_0$ is a limit ordinal.
It follows from the minimality of $\beta_0$ that
$\{u\upharpoonright\xi:\xi\in\beta_0\}\subseteq T$; therefore, as
$u\notin T$ and $\beta_0$ is a limit ordinal, we have that
$\{u\upharpoonright\xi:\xi\in\beta_0\}$ is a branch of $T$. Hence
$u=\bigcup\{u\upharpoonright\xi:\xi\in\beta_0\}\in L_T$. Note that
$f=\tilde u$, by the construction of $f$. It follows that $\sup A=u\in
L_T$.

\emph{Case 2.} $\beta_0$ is a successor ordinal.
Let $\alpha\in\omega_2$ be such that $\beta_0=\alpha+1$, and define
$v=u\upharpoonright\alpha\in T$. Since $u\notin T$ and $T$ satisfies
$(*)$, it follows from the construction of $f$ that
$u(\alpha)=0$. Thus $\{v\upharpoonright\xi:\xi\le\alpha\}$ is a
branch of $T$ (note that, by $(*)$, if $v^\smallfrown(0)\notin T$ then
$v$ has no extension in $T$), whence
$v=\bigcup\{v\upharpoonright\xi:\xi\le\alpha\}\in L_T$. As in the
previous case, $f=\tilde v$ and therefore $\sup A=v\in L_T$.
\end{proof}

\begin{lemma}
\label{pseudo}

$\psi(L_T)\le\aleph_1$.

\end{lemma}

\begin{proof}

Fix an arbitrary $s\in L_T$. It suffices to show that there is
$A\subseteq\{s'\in L_T:s'\prec s\}$ such that $|A|\le\aleph_1$ and
$(a,s)=\emptyset$, where $a=\sup A\in L_T$; a similar argument will
show that the analogous condition holds also to the right of $s$.

Let $S=\{\xi\in\mathrm{dom}(s):\exists s'\in L_T\;(s'\prec s$ and
$s'\upharpoonright\xi=s\upharpoonright\xi)\}$.

\emph{Case 1.} $S$ does not have a greatest element.
For each $\xi\in S$, pick $r_\xi\in L_T$ with $r_\xi\prec s$ and
$r_\xi\upharpoonright\xi=s\upharpoonright\xi$. It is clear that, if
$s'\in L_T$ is such that $s'\prec s$, then $s'\prec r_\xi$ for some
$\xi\in S$. Thus we can take $A=\{r_\xi:\xi\in S\}$.

\emph{Case 2.} There is $\beta=\max S$. Since
$\{s\upharpoonright\alpha:\alpha\in\mathrm{dom}(s)\}$ is a branch of
$T$, we must have that $\beta+1\in\mathrm{dom}(s)$; let then
$\theta=s(\beta+1)$. For each $\eta\in\theta$, let $r_\eta=\sup\{s'\in
L_T:s'\upharpoonright\beta=s\upharpoonright\beta$ and
$s'(\beta+1)=\eta\}\prec s$. Then $A=\{r_\eta:\eta\in\theta\}$
is as required.
\end{proof}

The next result will be the main tool for obtaining indestructibility
of our exam\-ples.

\begin{lemma}
\label{equiv}

The following conditions are equivalent:

\begin{itemize}

\item[$(a)$]
$L_T$ is destructible;

\item[$(b)$]
there is a countably closed forcing that adds a new Dedekind cut in
$\langle L_T,\prec\rangle$;

\item[$(c)$]
there is a countably closed forcing that adds a new branch
in $T$;

\item[$(d)$]
$T$ has a subtree isomorphic to $\langle\mbox{}^{<\omega_1}2,\subseteq\rangle$;

\item[$(e)$]
there is an $\omega_1$-\v Cech-Posp\'i\v sil tree in $L_T$.

\end{itemize}

\end{lemma}

\begin{proof}

Since $L_T$ is a linearly ordered compact space,
it fails to be compact in a generic extension by a countably closed
partial order if and only if this partial order adds in $L_T$ a new
Dedekind cut with no least upper bound. Thus the equivalence between
$(a)$ and $(b)$ follows from the observation that every Dedekind cut
in the extension that has a least upper bound was already in the
ground model.

For $(b)\rightarrow(c)$, let $A\subseteq L_T$ be a Dedekind cut
in a countably closed forcing extension, and suppose $A\notin\mathbf M$.
Let $\alpha$ be the least ordinal $\le\omega_2$ such that there is no
$t_\alpha\in T$ satisfying $\mathrm{dom}(t_\alpha)=\alpha$, $\{a\in
A:t_\alpha\subseteq a\}\neq\emptyset$ and $\{b\in L_T\setminus
A:t_\alpha\subseteq b\}\neq\emptyset$.
Note that $\{t_\beta:\beta\in\alpha\}$ is a chain in $T$; let then
$C=\{t_\beta:\beta\in\alpha\}$ and $u=\bigcup C$. We claim that
$u\notin\mathbf M$, which implies $C\notin\mathbf M$.

Suppose, to the contrary, that $u\in\mathbf M$.

\emph{Case 1.} $\alpha$ is a successor ordinal.
Then there is $\eta\in\omega_1+1$ such that $\{\xi\in\omega_1:\exists
a\in A\;(u\mbox{}^\smallfrown(\xi)\subseteq a)\}=\eta$ and
$u\mbox{}^\smallfrown(\eta)\subseteq b$ for some $b\in
L_T\setminus A$, hence $A=\{s\in L_T:s\prec
u\mbox{}^\smallfrown(\eta)\}\in\mathbf M$, a contradiction.

\emph{Case 2.} $\alpha$ is a limit ordinal.
Let $E=\{s\in L_T:u\subseteq s\}$. Then either $E\subseteq A$ or
$E\subseteq L_T\setminus A$, for otherwise we could have defined
$t_\alpha=u$. Therefore we have either $A=\{s\in L_T:\exists s'\in
L_T\;(u\subseteq s'$ and $s\preceq s')\}$ or
$A=\{s\in L_T:s\prec u\}$, thus contradicting the assumption that
$A\notin\mathbf M$.

Hence $u\notin\mathbf M$; in particular, this implies
$\mathrm{cf}(\alpha)>\omega$, since the forcing is countably
closed. Note that $C\notin\mathbf M$ must be a branch of $T$, for otherwise
we would have $u\in T$ and therefore $u\in\mathbf M$.

The implication $(c)\rightarrow(d)$ follows from essentially the same
argument present in the proof of Lemma 4.3 of \cite{stevotree}.

For $(d)\rightarrow(e)$, let
$\{t_p:p\in\mbox{}^{<\omega_1}2\}\subseteq T$ be such that $p\subseteq
q\leftrightarrow t_p\subseteq t_q$ for all
$p,q\in\mbox{}^{<\omega_1}2$. For each $p\in\mbox{}^{<\omega_1}2$,
define $F_p=L_T(t_p)$, which is closed by Lemma \ref{compl}. Now, for
each $h\in\mbox{}^{\omega_1}2$, let
$F_h=\bigcap\{F_{h\upharpoonright\alpha}:\alpha\in\omega_1\}$; note
that $F_h\neq\emptyset$ since $L_T$ is compact. Then
$\langle F_r:r\in\mbox{}^{\le\omega_1}2\rangle$ is an $\omega_1$-\v
Cech-Posp\'i\v sil tree in $L_T$.

Finally, $(e)\rightarrow(a)$ follows from Theorem \ref{st1} and
Proposition \ref{cptree}.
\end{proof}

Note that Lemma \ref{equiv} extends Corollary \ref{indestr}.

We now make the following observation, which is essentially taken from
Silver's Lemma (\cite{silver}; see also e.g. \cite[Lemma
VIII.3.4]{kunen}):

\begin{lemma}[Silver \cite{silver}]
\label{kurepano2}

If $T$ is an $\omega_1$-tree, then $T$ does not have a subtree
isomorphic to $\langle\mbox{}^{<\omega_1}2,\subseteq\rangle$.

\end{lemma}

\begin{proof}

Let $\langle T,\le\rangle$ be an $\omega_1$-tree, and suppose, to the
contrary, that $\{t_p:p\in\mbox{}^{<\omega_1}2\}\subseteq T$ is such
that $p\subseteq q\leftrightarrow t_p\le t_q$ for all
$p,q\in\mbox{}^{<\omega_1}2$. Let
$\theta=\sup\{\mathrm{ht}_T(t_p):p\in\mbox{}^{<\omega}2\}\in\omega_1$
and, for each $f\in\mbox{}^\omega2$, pick $u_f\in T_\theta$ such that
$t_f$ and $u_f$ are $\le$-comparable (this is possible since there
is $p\in\mbox{}^{<\omega_1}2$ extending $f$ with
$\mathrm{ht}_T(t_p)>\theta$). The mapping $f\mapsto u_f$ is
injective, which contradicts the fact that $T_\theta$ is countable.
\end{proof}

Now we can prove Theorem \ref{inaccBC}:

\begin{proof}[Proof of Theorem \ref{inaccBC}]

Let $T$ be a Kurepa tree with $\kappa>\aleph_1$ cofinal
branches. We may assume that $T$ is Hausdorff and rooted.
By Lemmas \ref{equiv} and \ref{kurepano2},
the linearly ordered compact space $L_T$ is
indestructible. Furthermore, $w(L_T)\le|T|=\aleph_1$ and
$|L_T|\ge\kappa>\aleph_1$.
\end{proof}

Here we remark that, although the space $L_T$ obtained from a Kurepa
tree is indestructible, the tree itself can of course fail to remain
Kurepa in a countably closed forcing extension, e.g. if the
forcing collapses $2^{\aleph_1}$ onto $\aleph_1$.

\begin{corol}
\label{equic.inacc}

The existence of an inaccessible cardinal and the statement ``every
Lindel\"of $T_3$ indestructible space of weight $\le\aleph_1$ has size
$\le\aleph_1$'' are equiconsistent.

\end{corol}

\begin{proof}

From Theorems \ref{BC} and \ref{inaccBC}.
\end{proof}

\begin{corol}
\label{equic.BC}

The $\aleph_1$-Borel Conjecture and the existence of an
inaccessible cardinal are equiconsistent.

\end{corol}

\begin{proof}

Since Tychonoff spaces of weight $\le\aleph_1$ are embeddable in
$[0,1]^{\omega_1}$, the result follows from Theorems \ref{proj} and
\ref{inaccBC}.
\end{proof}

We also have:

\begin{corol}
\label{2^w1}

For a regular cardinal $\kappa>\aleph_1$, it is consistent with
$\mathrm{ZFC}$ that
there is a
compact $T_2$ indestructible space of weight $\le\aleph_1$ and size
$2^{\aleph_1}=\kappa$.

\end{corol}

\begin{proof}

Since it is consistent that $2^{\aleph_1}=\kappa$ + ``there is a
Kurepa tree with $\kappa$ cofinal branches'' (take e.g. a model of
$\mathrm{BACH}$ + $2^{\aleph_1}=\kappa^+$ as in Theorem 8.3 of
\cite{tallbach} and collapse $\kappa^+$ to $\kappa$ using
$\kappa$-closed forcing), we can proceed as in the proof of Theorem
\ref{inaccBC}.
\end{proof}

Now we turn to Theorem \ref{weak}.
The main ingredient needed in the proof we shall present for it is the
following combination of Theorem 6.1 of \cite{jensen} (see
1.10 in \cite{stevo*}) and Theorem 3.9 of \cite{koenig}:

\begin{thm}[Jensen \cite{jensen}, K\"onig \cite{koenig}]
\label{jk}

If a regular uncountable cardinal $\kappa$ is not weakly compact in
$\mathbf L$, then 
there is a sequence
$\mathcal F=(f_\alpha)_{\alpha\in\lim(\kappa)}$ of functions
$f_\alpha:\alpha\rightarrow\alpha$ that is \emph{coherent} ---
i.e. satisfies $f_\alpha=^* f_\beta\!\upharpoonright\!\alpha$
for every $\alpha,\beta\in\lim(\kappa)$ with $\alpha<\beta$ --- and
such that $\langle T(\mathcal F),\subseteq\rangle$ is a $\kappa$-Aronszajn tree,
where $T(\mathcal
F)=\bigcup_{\xi\in\kappa}\bigcup_{\alpha\in\lim(\kappa)\setminus\xi}
\{f\in\mbox{}^\xi\xi:f=^*f_\alpha\!\upharpoonright\!\xi\}$.

\end{thm}

The following lemma will be essential in what we shall do next.

\begin{lemma}
\label{cfw}

If $\mathcal F=(f_\alpha)_{\alpha\in\lim(\omega_2)}$ is coherent, then
no branch of $\langle T(\mathcal F),\subseteq\rangle$ has cofinality $\omega_1$.

\end{lemma}

\begin{proof}

Suppose, to the contrary, that $B\subseteq T(\mathcal F)$ is a branch
of cofinality $\omega_1$. Let $g=\bigcup B$ and
$\gamma=\mathrm{dom}(g)\in\lim(\omega_2)$, and fix a strictly
increasing sequence of limit ordinals
$(\gamma_\eta)_{\eta\in\omega_1}$ with
$\sup\{\gamma_\eta:\eta\in\omega_1\}=\gamma$.
Note that for each $\eta\in\omega_1$ we have
$g\!\upharpoonright\!\gamma_\eta\in T(\mathcal F)$, and thus
$g\!\upharpoonright\!\gamma_\eta=^* f_{\gamma_\eta}=^*
f_\gamma\!\upharpoonright\!\gamma_\eta$ since $\mathcal F$ is
coherent. This yields
$$
\omega_1=\bigcup_{k\in\omega}\{\eta\in\omega_1:|\{\xi\in\gamma_\eta:g(\xi)\neq
f_\gamma(\xi)\}|=k\};
$$
thus, by regularity of $\omega_1$, there is $k_0\in\omega$ such that
$|\{\xi\in\gamma_\eta:g(\xi)\neq f_\gamma(\xi)\}|=k_0$ for uncountably
many $\eta\in\omega_1$.
Since the mapping $\eta\mapsto|\{\xi\in\gamma_\eta:g(\xi)\neq
f_\gamma(\xi)\}|$ is non-decreasing, we must have
$|\{\xi\in\gamma_\eta:g(\xi)\neq f_\gamma(\xi)\}|=k_0$ for every
$\eta\in\omega_1$ greater than some $\eta_0\in\omega_1$. But then
$|\{\xi\in\gamma:g(\xi)\neq f_\gamma(\xi)\}|=k_0$, which implies
$g\in T(\mathcal F)$, thus contradicting the choice of $B$.
\end{proof}

The next lemma will guarantee indestructibility of the space we shall
obtain.

\begin{lemma}
\label{no2w1}

Assume $\mathrm{CH}$.
If $\mathcal F=(f_\alpha)_{\alpha\in\lim(\omega_2)}$ is coherent, then
$\langle T(\mathcal F),\subseteq\rangle$ does not have a subtree
isomorphic to $\langle\mbox{}^{<\omega_1}2,\subseteq\rangle$.

\end{lemma}

\begin{proof}

This is pretty much like Lemma \ref{kurepano2}.
Suppose that there is $\{g_s:s\in\mbox{}^{<\omega_1}2\}\subseteq
T(\mathcal F)$ such that $g_s\subseteq g_t\leftrightarrow s\subseteq
t$ for all $s,t\in\mbox{}^{<\omega_1}2$. By $\mathrm{CH}$, we have
that
$\delta=\sup\{\mathrm{dom}(g_s):s\in\mbox{}^{<\omega_1}2\}\in\omega_2$.
For each $h\in\mbox{}^{\omega_1}2$, let
$g_h=\bigcup\{g_{h\upharpoonright\alpha}:\alpha\in\omega_1\}$, and
then define $\tilde
g_h=g_h\cup(f_\delta\!\upharpoonright(\delta\setminus\mathrm{dom}(g_h)))$;
note that $g_h\in T(\mathcal F)$ by Lemma \ref{cfw}, and thus $\tilde
g_h\in T(\mathcal F)$.
But
the $\delta$-th level of $T(\mathcal F)$ is the set
$\{f\in\mbox{}^\delta\delta:f=^*f_\delta\}$, which has cardinality
$\aleph_1$; this leads to a contradiction, since this set must include
$\{\tilde g_h:h\in\mbox{}^{\omega_1}2\}$ and $h\mapsto\tilde g_h$ is
one-to-one.
\end{proof}

The following proposition will then be the core of our proof of
Theorem \ref{weak}:

\begin{prop}
\label{weak0}

If $\mathrm{CH}$ holds and $\omega_2$ is not weakly compact in
$\mathbf L$, then there is a compact $T_2$ indestructible space of
pseudocharacter $\aleph_1$ and size $\aleph_2$.

\end{prop}

\begin{proof}

Assume that $\omega_2$ is not weakly compact in $\mathbf L$, and let
then $\mathcal F=(f_\alpha)_{\alpha\in\lim(\omega_2)}$ be given by
Theorem \ref{jk}. Now consider the linearly ordered compact space
$L_T$, where $T$ is a tree isomorphic to $T(\mathcal F)$ that
satisfies the assumptions stated in the paragraph preceding Lemma
\ref{compl}. By Lemmas \ref{equiv} and \ref{no2w1},
$L_T$ is indestructible. On the one hand, the fact that $T$ is
$\omega_2$-Aronszajn implies that
$|L_{T}|\ge\aleph_2$; on the other hand, in view of Lemma \ref{cfw},
it also implies that eve\-ry branch of $T$ has countable cofinality,
and thus $L_{T}\subseteq\{\bigcup C:C\in[T]^{\le\aleph_0}\}$; since
$|T|=\aleph_2$, this yields
$|L_{T}|\le\aleph_2\mbox{}^{\aleph_0}=\aleph_1\mbox{}^{\aleph_0}\cdot\aleph_2=\aleph_2$
by $\mathrm{CH}$. Therefore, $|L_{T}|=\aleph_2$.
Finally, $\psi(L_{T})\le\aleph_1$ by Lemma \ref{pseudo},
and so $\psi(L_{T})=\aleph_1$ since otherwise we would
contradict Arhangelski\u\i 's Theorem \cite{arh}.
\end{proof}

Now Theorem \ref{weak} follows directly from Proposition \ref{weak0}:

\begin{proof}[Proof of Theorem \ref{weak}]

If $\mathrm{CH}$ fails, then any subspace $X\subseteq\mathbb R$ with
$|X|=\aleph_2$ will sa\-tis\-fy the required conditions (recall that, by
Theorem 6(a) of \cite{tall95}, every hereditarily Lindel\"of space is
indestructible). If $\mathrm{CH}$ holds, the result follows from
Proposition \ref{weak0}.
\end{proof}

\begin{corol}
\label{equic.weak}

The existence of a weakly compact cardinal and the statement
``there is no Lindel\"of $T_1$ space of pseudocharacter
$\le\aleph_1$ and size $\aleph_2$'' are equiconsistent.

\end{corol}

\begin{proof}

From Theorems \ref{psi} and \ref{weak}.
\end{proof}

We conclude with the following observation:

\begin{prop}
\label{final}

Assume that $\kappa$ is an inaccessible cardinal that is not weakly
compact in $\mathbf L$.
If $\kappa$ is L\'evy-collapsed to $\aleph_2$,
the model $\mathbf M[G]$ thus obtained satisfies:
\begin{itemize}
\item[$(i)$]
every Lindel\"of $T_3$ indestructible space of
weight $\le\aleph_1$ has size $\le\aleph_1$;
and
\item[$(ii)$]
there is a compact $T_2$ indestructible space with pseudocharacter
$\aleph_1$ and size $\aleph_2$.
\end{itemize}

\end{prop}

\begin{proof}

Theorem \ref{BC} is proven in \cite{usuba} by showing that $(i)$ holds
in the model obtained from the L\'evy collapse of an inaccessible to
$\aleph_2$. Thus we only need to check $(ii)$.
Note that ``$\omega_2$ is not weakly compact in $\mathbf L$'' holds in
$\mathbf M[G]$, since the absoluteness of the definition of $\mathbf
L$ implies that $\mathbf L^{\mathbf M}=\mathbf L^{\mathbf M[G]}$.
As $\mathrm{CH}$ also holds in $\mathbf M[G]$, we can then apply
Proposition \ref{weak0}.
%
%
%
\end{proof}

\section*{Acknowledgements}
We would like to thank L\'ucia R. Junqueira for her valuable comments
and helpful discussions during the preparation of this paper. We would
also like to thank the referee for the useful suggestions.

\end{document}